\documentclass{amsart}
\usepackage{amssymb}
\newtheorem{Theo}{Theorem}
\newtheorem{Lem}{Lemma}
\newtheorem{Cor}{Corollary}

\begin{document}

\title{A note on the asymptotics for incomplete Betafunctions
}


\author{Jan-Christoph Schlage-Puchta
}




\maketitle

\begin{abstract}
We determine the asymptotic behaviour of certain incomplete Betafunctions.
\end{abstract}

\section{Introdution and results}
For integers $k,\ell\geq 1$ define 
\[
P_{k,\ell} = \frac{(k+\ell)!}{(k-1)!\ell!}\int_{k/(k+\ell)}^1 t^{k-1}(1-t)^\ell\;dt.
\]
By repeated partial integration we obtain the representation
\[
P_{k, \ell} = \left(\frac{\ell}{\ell+k}\right)^{k+\ell}\sum_{\nu=0}^{k-1}\binom{k+\ell}{\nu}\left(\frac{k}{\ell}\right)^\nu.
\]
Vietoris\cite{Vietoris} showed that $P_{k,\ell}\leq\frac{1}{2}$. Alzer and Kwong \cite{Alzer} proved that $P_{k,\ell}\geq\frac{1}{4}$ holds for all $k, \ell$. Interest in bounds for $P_{k,\ell}$ stems from application to statistics, see \cite{Uhlmann}.

In this note we consider the asymptotic behaviour of $P_{k,\ell}$. We prove the following.
\begin{Theo}
\label{thm:main}
We have
\begin{equation}
\label{eq:k small}
P_{k, l} = e^{-k}\sum_{\nu=0}^{k-1}\frac{k^\nu}{\nu!} + \mathcal{O}\left(\frac{k^2}{\ell}\right),
\end{equation}
\begin{equation}
\label{eq:ell small}
P_{k,\ell} = 1-e^{-\ell}\sum_{\nu=0}^{\ell}\frac{\ell^\nu}{\nu!} + \mathcal{O}\left(\frac{\ell^2}{k}\right)
\end{equation}
and
\begin{equation}
\label{eq:general}
P_{k, \ell} = \frac{1}{2}+\mathcal{O}\left(\frac{1}{\sqrt{\min(k, \ell)}}\right).
\end{equation}
\end{Theo}
Note that the first two estimates are good if one of the parameters $k, \ell$ is rather small, whereas the third one gives information in the general case.

Comparing (\ref{eq:k small}) and (\ref{eq:ell small}) with Vietoris' result that $P_{k,\ell}\leq\frac{1}{2}$ we see that
\[
\sum_{\nu=0}^{\ell-1}\frac{\ell^\nu}{\nu!}<\frac{1}{2}e^\ell<\sum_{\nu=0}^{\ell}\frac{\ell^\nu}{\nu!}.
\]
Note that equality is impossible as $e$ is transcendental.
A more precise version of this inequality has been asked by Ramanujan (Question 294) and was answered by Karamata~\cite{Karamata}. For a detailed discussion of this result, Uhlmann's inequalities \cite{Uhlmann} and Vietoris bound we refer the reader to the historical notes by Vietoris \cite{V2}.

From Theorem~\ref{thm:main} we deduce the following.
\begin{Cor}
\label{Cor:accumulation}
The only accumulation points of the set $\{P_{k,\ell}:k, \ell\in\mathbb{N}\}$ are $\frac{1}{2}$ and the real numbers $e^{-k}\sum_{\nu=1}^{k-1}\frac{k^\nu}{\nu!}$ and $1-e^{-\ell}\sum_{\nu=1}^{\ell}\frac{\ell^\nu}{\nu!}$.
\end{Cor}
\begin{proof}
Suppose that $(k_n), (\ell_n)$ are integer sequences such that the pairs $(k_n, \ell_n)$ are all distinct and that the sequence $(P_{k_n, \ell_n})$ converges. If both $k_n, \ell_n$ tend to infinity, then by (\ref{eq:general}) we have that $P_{k_n, \ell_n}$ tends to $\frac{1}{2}$. If one of these sequences does not tend to infinity, then we can pass to a subsequence and assume that $k_n$ or $\ell_n$ is constant. Then our claim follows from (\ref{eq:k small}) or (\ref{eq:ell small}).
\end{proof}
Together with Vietoris bound we obtain the following.
\begin{Cor}
We have
\[
\lim_{\ell\rightarrow\infty} P_{k, \ell} \geq \frac{1}{2} - \frac{1}{\sqrt{2\pi k}}.
\]
\end{Cor}
\begin{proof}
We have
\[
\lim_{\ell\rightarrow\infty}\left(P_{k, \ell}+P_{\ell, k}\right) = 1-\frac{k^k}{e^k k!},
\]
using Stirling's formula in the form $k!=\left(\frac{k}{e}\right)^k\sqrt{2\pi k}e^{\theta/12 k}$ with $0\leq\theta\leq 1$, and $P_{\ell, k}\leq \frac{1}{2}$, we conclude
\[
\lim_{\ell\rightarrow\infty} P_{k, \ell} \geq \frac{1}{2}-\frac{k^k}{e^k k!}\geq \frac{1}{2}-\frac{1}{\sqrt{2\pi k}}.
\]
\end{proof}

\section{Preliminary estimates}

For the proof of (\ref{eq:general}) we use another representation of $P_{k,\ell}$, which is due to Raab \cite{Raab}.
\begin{Theo}
\label{thm:Raab}
We have $P_{k,\ell}=U_{k,\ell} V_{k,\ell}$, where
\[
U_{k,\ell} = \exp\int_0^\infty \frac{1}{t}\left(\frac{1}{e^t-1}-\frac{1}{t}+\frac{1}{2}\right)\left(e^{-(k+\ell)t}-e^{-kt}-e^{-\ell t}\right)\;dt,
\]
\[
V_{k, \ell} = \frac{\sqrt{\ell}}{2\pi}\sum_{\nu=1}^\infty\frac{c_\nu(k/\ell)}{\sqrt{\nu}(\nu+\ell)},
\]
and
\[
c_\nu(x) = \exp\int_0^\infty \frac{1}{t}\left(\frac{1}{e^t-1}-\frac{1}{t}+\frac{1}{2}\right)\left(e^{-\nu(1+x)t}-e^{-\nu xt}-e^{-\nu t}\right)\;dt.
\]
\end{Theo}
We first compute the occurring integrals.
\begin{Lem}
\label{Lem:integrals}
We have for $x\geq 1$
\[
\int_0^\infty \frac{1}{t}\left(\frac{1}{e^t-1}-\frac{1}{t}+\frac{1}{2}\right) e^{-xt}\;dt =  \frac{1}{12x} + \mathcal{O}\left(\frac{1}{x^2}\right)
\]
and
\[
\int_0^\infty \frac{1}{t}\left(\frac{1}{e^t-1}-\frac{1}{t}+\frac{1}{2}\right) e^{-xt}\;dt = \frac{1}{2}\log x+ C + \mathcal{O}(x)
\]
for $0<x\leq 1$, where $C$ is some constant. 
\end{Lem}
\begin{proof}
The series expansion of $e^x$ yields for $t\rightarrow 0$
\[
\frac{1}{t}\left(\frac{1}{e^t-1}-\frac{1}{t}+\frac{1}{2}\right) = \frac{1}{t^2}\left(\frac{1}{1+t/2+t^2/6+\mathcal{O}(t^3)}-1+\frac{t}{2}\right) = 1/12 + \mathcal{O}(t).
\]
For $t\rightarrow\infty$ this expression tends to 0, in particular, it is bounded for all positive $t$. Hence for $x\rightarrow\infty$ the integral in question becomes
\[
\frac{1}{12}\int_0^\infty e^{-xt}\;dt + \mathcal{O}\left(\int_0^\infty te^{-xt}\;dt\right) = \frac{1}{12 x} + \mathcal{O}\left(\frac{1}{x^2}\right).
\]
For $x\rightarrow 0$ we use $e^{-xt}=1-xt + \mathcal{O}(x^2t^2)$ and obtain
\[
\int_0^1 \frac{1}{t}\left(\frac{1}{e^t-1}-\frac{1}{t}+\frac{1}{2}\right) e^{-xt}\;dt = \int_0^1 \frac{1}{t}\left(\frac{1}{e^t-1}-\frac{1}{t}+\frac{1}{2}\right)\;dt + \mathcal{O}(x)
=C_1+\mathcal{O}(x)
\]
and 
\[
\int_1^\infty \frac{1}{t(e^t-1)} e^{-xt}\;dt = \int_1^\infty \frac{1}{t(e^t-1)} \;dt -x \int_1^\infty \frac{1}{e^t-1} \;dt + \mathcal{O}(x^2).
\]
As a function of $x$, the integral $\int_1^{\infty} \frac{e^{-xt}}{t^2}\;dt$ defines a function that is differentiable from the right in 0 and has bounded second derivative in $(0,\infty)$, hence, for $x\geq 0$ this integral is $1+\mathcal{O}(x)$. 

Finally 
\begin{multline*}
\int_1^\infty\frac{e^{-xt}}{2t}\;dt = \int_x^\infty\frac{e^{-s}}{2s}\;ds =\\
 \int_x^1\frac{dt}{2t} + \int_0^1\frac{e^{-t}-1}{2t}\;dt +\int_1^\infty\frac{e^{-t}}{2t}\;dt+\int_0^x\frac{1-e^{-t}}{2t}\;dt\\
  = -\frac{1}{2}\log x + C_3 + \mathcal{O}(x).
\end{multline*}
Combining these estimates our claim follows.
\end{proof}
From this we obtain
\begin{Lem}
\label{Lem:U}
We have
\[
U_{k, \ell} =1-\frac{1}{12}\left(\frac{1}{k}+\frac{1}{\ell}-\frac{1}{k+\ell}\right)+\mathcal{O}\left(\frac{1}{\min(k, \ell)^2}\right).
\]
\end{Lem}
\begin{proof}
From Lemma~\ref{Lem:integrals} we obtain
\[
U_{k, \ell} = \exp\left(\frac{1}{12}\left(\frac{1}{k+\ell}-\frac{1}{k}-\frac{1}{\ell}\right)+\mathcal{O}\left(\frac{1}{\min(k, \ell)^2}\right)\right),
\]
inserting the Taylor series for $\exp$ and using the fact that $k, \ell\geq 1$ our claim follows.
\end{proof}
Next we compute $c_\nu(x)$.
\begin{Lem}
\label{Lem:c}
If $\nu x\geq 1$, then
\[
c_\nu(x)=1+\frac{1}{12\nu(1+x)}-\frac{1}{12\nu x}+\frac{1}{12\nu} + \mathcal{O}\left(\frac{1}{\nu^2 \min(1,x^2)}\right).
\]
If $\nu x\leq 1$, then
\[
c_\nu(x) = K\sqrt{\nu x} + \mathcal{O}\left(\sqrt{\frac{x}{\nu}}+(\nu x)^{3/2}\right)
\]
for some constant $K$.
\end{Lem}
\begin{proof}
If $\nu x>1$, then we apply Lemma~\ref{Lem:integrals} to obtain
\begin{eqnarray*}
c_\nu(x) & = & \exp\left(\frac{1}{12\nu(x+1)}-\frac{1}{12\nu x}+\frac{1}{12\nu}+\mathcal{O}\left(\frac{1}{\nu^2\min(1, x^2)}\right)\right)\\
 & = & 1 + \frac{1}{12\nu(x+1)}-\frac{1}{12\nu x}+\frac{1}{12\nu}+\mathcal{O}\left(\frac{1}{\nu^2\min(1, x^2)}\right).
\end{eqnarray*}
If $\nu x<1$, then $\nu(x+1)\geq 1$, and we obtain
\begin{eqnarray*}
c_\nu(x) & = & \exp\left(\frac{1}{12\nu(x+1)}+\frac{1}{2}\log(\nu x)+C+\frac{1}{12\nu}+\mathcal{O}\left(\frac{1}{\nu^2}+\nu x\right)\right)\\
 & = & K\sqrt{\nu x} + \mathcal{O}\left(\sqrt{\frac{x}{\nu}}+(\nu x)^{3/2}\right).
\end{eqnarray*}
\end{proof}

We now compute $V_{k,\ell}$.

\begin{Lem}
\label{Lem:V}
We have
\[
V_{k,\ell} = \frac{1}{2} + \mathcal{O}\left(\frac{1}{\sqrt{k}}+\frac{1}{\sqrt{\ell}}\right)
\]
\end{Lem}
\begin{proof}
We have
\[
\sum_{\nu\leq \ell/k}\frac{c_\nu(k/\ell)}{\sqrt{\nu}(\nu+\ell)}\ll \sum_{\nu\leq \ell/k}\frac{\sqrt{\nu k/\ell}}{\sqrt{\nu}(\nu+\ell)}\ll \frac{1}{\sqrt{\ell k}}
\] 
thus,
\begin{eqnarray*}
V_{k,\ell} & = & \frac{\sqrt{\ell}}{2\pi}\sum_{\nu>\ell/k}\frac{1}{\sqrt{\nu}(\nu+\ell)}\left(1+\frac{1}{12\nu\frac{k+\ell}{\ell}} - \frac{1}{12\nu\frac{k}{\ell}} + \frac{1}{12\nu}+\mathcal{O}\left(\frac{1}{\nu^2\min(1, k^2/\ell^2)}\right)\right)+ \mathcal{O}\left(\frac{1}{\sqrt{k}}\right)\\
 & = & \frac{\sqrt{\ell}}{2\pi}\sum_{\nu>\ell/k}\frac{1}{\sqrt{\nu}(\nu+\ell)}\left(1 - \frac{1}{12\nu\frac{k}{\ell}}+\mathcal{O}\left(\frac{1}{\nu^2\min(1, k^2/\ell^2)}\right)\right)+ \mathcal{O}\left(\frac{1}{\sqrt{k}}+\sum_{\nu>\ell/k}\frac{1}{\nu^{3/2}\sqrt{\ell}}\right)\\ 
 & = & \frac{\sqrt{\ell}}{2\pi}\sum_{\nu>\ell/k}\frac{1}{\sqrt{\nu}(\nu+\ell)}\left(1 - \frac{1}{12\nu\frac{k}{\ell}}+\mathcal{O}\left(\frac{1}{\nu^2\min(1, k^2/\ell^2)}\right)\right)+ \mathcal{O}\left(\frac{1}{\sqrt{k}}+\frac{1}{\sqrt{\ell}}\right)\\
\end{eqnarray*}

If $k>\ell$ we obtain
\begin{eqnarray*}
V_{k,\ell} & = & \frac{\sqrt{\ell}}{2\pi}\sum_{\nu=1}^\infty\frac{1}{\sqrt{\nu}(\nu+\ell)} + \mathcal{O}\left(\frac{1}{\sqrt{k}}+\frac{1}{\sqrt{\ell}}+\sqrt{\ell}\sum_{\nu=1}^\infty\frac{1}{\nu^{3/2}(\nu+\ell)}\right)\\
 & = & \frac{1}{2\pi}\int_0^\infty\frac{dt}{\sqrt{t}(t+1)} + \mathcal{O}\left(\frac{1}{\sqrt{k}}+\frac{1}{\sqrt{\ell}}\right)\\
  & = & \frac{1}{2} + \mathcal{O}\left(\frac{1}{\sqrt{\ell}}\right).
\end{eqnarray*}
We have
\[
\frac{\sqrt{\ell}}{2\pi}\sum_{\nu>\ell}\frac{1}{\sqrt{\nu}(\nu+\ell)}\left(1 - \frac{1}{12\nu\frac{k}{\ell}}+\mathcal{O}\left(\frac{1}{\nu^2\min(1, k^2/\ell^2)}\right)\right) = \frac{\sqrt{\ell}}{2\pi}\sum_{\nu>\ell}\frac{1}{\sqrt{\nu}(\nu+\ell)}+\mathcal{O}\left(\frac{1}{\sqrt{k}}\right),
\]
and for $\ell/k\leq N\leq \ell$
\[
\sum_{N\leq \nu\leq 2N}\frac{1}{\nu^{3/2}(\nu+\ell)} \leq \frac{1}{\sqrt{N}\ell}
\]
as well as
\[
\sum_{N\leq \nu\leq 2N}\frac{1}{\nu^{5/2}(\nu+\ell)} \leq \frac{1}{N^{3/2}\ell}
\]
and therefore
\[
\sqrt{\ell}\sum_{\nu\geq\ell/k}\frac{1}{\sqrt{\nu}(\nu+\ell)}\cdot\frac{1}{\nu\frac{k}{\ell}} = \frac{\ell^{3/2}}{k}\sum_{\nu\geq \ell/k}\frac{1}{\nu^{3/2}(\nu+\ell)} =\mathcal{O}\left(\frac{1}{\sqrt{k}}\right)
\]
and
\[
\sqrt{\ell}\sum_{\nu\geq \ell/k} \frac{1}{\sqrt{\nu}(\nu+\ell)}\cdot\frac{1}{\nu^2\min(k^2/\ell^2)} = \frac{\ell^{5/2}}{k^2}\sum_{\nu\geq \ell/k}\frac{1}{\nu^{5/2}(\nu+\ell)} =\mathcal{O}\left(\frac{1}{\sqrt{k}}\right).
\]
Using these estimates we obtain
\begin{multline*}
\frac{\sqrt{\ell}}{2\pi}\sum_{\ell/k\leq\nu\leq\ell}\frac{1}{\sqrt{\nu}(\nu+\ell)}\left(1 - \frac{1}{12\nu\frac{k}{\ell}}+\mathcal{O}\left(\frac{1}{\nu^2\min(1, k^2/\ell^2)}\right)\right) =\\
\frac{\sqrt{\ell}}{2\pi}\sum_{\ell/k\leq\nu\leq\ell}\frac{1}{\sqrt{\nu}(\nu+\ell)} +\mathcal{O}\left(\frac{1}{\sqrt{k}}\right).
\end{multline*}
For $k\leq \ell$ we obtain
\begin{eqnarray*}
V_{k,\ell} & = & \frac{\sqrt{\ell}}{2\pi}\sum_{\nu\geq \ell/k}\frac{1}{\sqrt{\nu}(\nu+\ell)} +\mathcal{O}\left(\frac{1}{\sqrt{k}}\right)\\
 & = & \frac{1}{2\pi}\int_{1/k}^\infty\frac{dt}{\sqrt{t}(t+1)} +\mathcal{O}\left(\frac{1}{\sqrt{k}}\right)\\
 & = & \frac{1}{2\pi}\int_0^\infty\frac{dt}{\sqrt{t}(t+1)} +\mathcal{O}\left(\frac{1}{\sqrt{k}}\right)\\
 & = & \frac{1}{2}+\mathcal{O}\left(\frac{1}{\sqrt{k}}\right),
\end{eqnarray*}
and the proof is complete.
\end{proof}

\section{Proof of Theorem~\ref{thm:main}}

We first prove (\ref{eq:k small}). Note that this inequality is only interesting if the error term is $o(1)$, in particular we may assume that $k^2<\ell$. Under this assumption we have
\begin{eqnarray*}
P_{k, \ell} & = & \left(\frac{\ell}{\ell+k}\right)^{k+\ell}\sum_{\nu=0}^{k-1}\binom{k+\ell}{\nu}\left(\frac{k}{\ell}\right)^\nu\\
 & = & \sum_{\nu=0}^{k-1}\left(\frac{\ell}{\ell+k}\right)^{k+\ell-\nu}\frac{k^\nu}{\nu!} \exp\left(\sum_{\mu=1}^\nu\log\frac{k+\ell-\mu}{k+\ell}\right)\\
 & = & \sum_{\nu=0}^{k-1}\left(\frac{\ell}{\ell+k}\right)^{k+\ell-\nu}\frac{k^\nu}{\nu!} \exp\left(-\sum_{\mu=1}^\nu\frac{\mu}{k+\ell}+ \mathcal{O}\left(\frac{\mu^2}{(k+\ell)^2}\right)\right)\\
 & = & \sum_{\nu=0}^{k-1}\left(\frac{\ell}{\ell+k}\right)^{k+\ell-\nu}\frac{k^\nu}{\nu!} \exp\left(\mathcal{O}\left(\frac{\nu^2}{\ell}\right)\right)\\
 & = & \sum_{\nu=0}^{k-1}\left(\frac{\ell}{\ell+k}\right)^{k+\ell}\frac{k^\nu}{\nu!} \exp\left(\mathcal{O}\left(\frac{k^2}{\ell}\right)\right)\\ 
 & = & \exp\left((k+\ell)\log\frac{\ell}{\ell+k}\right) \sum_{\nu=0}^{k-1}\frac{k^\nu}{\nu!} \exp\left(\mathcal{O}\left(\frac{k^2}{\ell}\right)\right)\\
 & = & \exp\left(-k+\mathcal{O}\left(\frac{k^2}{\ell}\right)\right) \sum_{\nu=0}^{k-1}\frac{k^\nu}{\nu!}\\
 & = & e^{-k}\sum_{\nu=0}^{k-1}\frac{k^\nu}{\nu!}+\mathcal{O}\left(\frac{k^2}{\ell}\right).
\end{eqnarray*}
The proof of (\ref{eq:ell small}) is quite similar. Analogously to the previous case we may assume that $\ell^2<k$.

\begin{eqnarray*}
P_{k, \ell}  & = & \left(\frac{\ell}{\ell+k}\right)^{k+\ell}\left(\left(\frac{\ell+k}{\ell}\right)^{k+\ell} - \sum_{\nu=k}^{k+\ell}\binom{k+\ell}{\nu}\left(\frac{k}{\ell}\right)^\nu\right)\\
 & = & 1-\left(\frac{\ell}{\ell+k}\right)^{k+\ell} \sum_{\mu=0}^{\ell}\binom{k+\ell}{\mu}\left(\frac{k}{\ell}\right)^{k+\ell-\mu}\\
 & = & 1-\left(\frac{k}{\ell+k}\right)^{k+\ell} \sum_{\mu=0}^{\ell}\binom{k+\ell}{\mu}\left(\frac{\ell}{k}\right)^{\mu}\\
 & = & 1-\exp\left(-\ell+\mathcal{O}\left(\frac{\ell^2}{k}\right)\right)\sum_{\mu=0}^{\ell}\frac{\ell^\mu}{\mu!}\left(\frac{k+\ell}{k}\right)^{\mu}\\
 & = & 1-e^{-\ell}\sum_{\mu=0}^{\ell}\frac{\ell^\mu}{\mu!}.
\end{eqnarray*}

Finally (\ref{eq:general}) follows from Theorem~\ref{thm:Raab}, Lemma~\ref{Lem:U} and Lemma~\ref{Lem:V}.


\begin{thebibliography}{9}
\bibitem{Alzer} H. Alzer, M. K.  Kwong, 
Inequalities for combinatorial sums, {\em Arch. Math.} {\bf 108} (2017), 601–607.
\bibitem{Karamata} J. Karamata, Sur quelques probl\`emes pos\'es par Ramanujan,
{\em J. Indian Math. Soc.} (N.S.) {\bf 24} (1960), 343–365.  
\bibitem{Raab} W. Raab, Die Ungleichungen von Vietoris, {\em 
Monatsh. Math.} {\bf 98} (1984), 311–322. 
\bibitem{Uhlmann} W. Uhlmann,
{\em Statistische Qualit\"atskontrolle: Eine Einf\"uhrung}
Leitf\"aden der angewandten Mathematik und Mechanik Band 7. B. G. Teubner Verlagsgesellschaft, Stuttgart 1966. 
\bibitem{Vietoris} L. Vietoris, \"Uber gewisse die unvollst\"andige Betafunktion betreffende Ungleichungen, {\em \"Os\-ter\-reich. Akad. Wiss. Math.-Natur. Kl. Sitzungsber. II} {\bf 191} (1982), 85–92. 
\bibitem{V2} L. Vietoris, Geschichtliches \"uber gewisse Ungleichungen,
{\em \"Osterreich. Akad. Wiss. Math.-Natur. Kl. Sitzungsber. II} {\bf 193} (1984), 
319–321. 
\end{thebibliography}
\end{document}